\documentclass{article}
\usepackage{amssymb}
\usepackage{amsmath}
\usepackage{epsfig,graphics,graphicx, color,float,url}
\usepackage[numbers,sort&compress]{natbib}
\bibpunct{[}{]}{,}{n}{}{;}
\newtheorem{theorem}{Theorem}

\newtheorem{corollary}[theorem]{Corollary}

\newtheorem{definition}[theorem]{Definition}

\newtheorem{lemma}[theorem]{Lemma}

\newtheorem{proposition}[theorem]{Proposition}
\newtheorem{remark}[theorem]{Remark}

\newenvironment{proof}[1][Proof]{\noindent\textbf{#1.} }{\ \rule{0.5em}{0.5em}}

\newcommand{\RR}{\mathbb{R}}

\renewcommand{\a}{\alpha}

\newcommand{\e}{\varepsilon}

\newcommand{\D}{\Delta}

\begin{document}

\title{Jensen-type inequalities for convex and $m$-convex functions via fractional calculus}

\author{\textbf{Quintana, Yamilet}\\Universidad Carlos III de Madrid, Departamento de Matem\'aticas\\Avenida de la Universidad 30, 28911 Legan\'{e}s, Madrid, Spain\\\texttt{yaquinta@math.uc3m.es}\\
\and
\textbf{Rodr\'{\i}guez, Jos\'{e} M.}\\Universidad Carlos III de Madrid, Departamento de Matem\'aticas\\Avenida de la Universidad 30, 28911 Legan\'{e}s, Madrid, Spain\\\texttt{jomaro@math.uc3m.es}\\
\and \textbf{Sigarreta Almira, Jos\'{e} M.}\\Universidad Aut\'{o}noma de Guerrero, Centro Acapulco\\CP 39610, Acapulco de Ju\'{a}rez, Guerrero, M\'exico\\\texttt{josemariasigarretaalmira@hotmail.com} }

\maketitle

\begin{abstract}
Inequalities play an important role in pure and applied mathematics.
In particular, Jensen's inequality, one of the most famous inequalities,
plays a main role in the study of the existence and uniqueness of initial and boundary value
problems for differential equations.
In this work we prove some new Jensen-type inequalities  for $m$-convex functions, and we apply them to generalized Riemann-Liouville-type
integral operators.
It is remarkable that, if we consider $m=1$, we obtain new inequalities for convex functions.
\end{abstract}

\bigskip



\bigskip

\section{Introduction}

Integral inequalities are used in countless
mathematical problems such as approximation theory and spectral analysis, statistical analysis and the
theory of distributions. Studies involving integral inequalities play an important role in several areas
of science and engineering.

In recent years there has been a growing interest in the study of many classical inequalities applied to integral operators associated with different types of fractional derivatives,
since integral inequalities and their applications
play a vital role in the theory of differential equations and applied mathematics.
Some of the inequalities studied are Gronwall, Chebyshev, Hermite-Hadamard-type, Ostrowski-type, Gr\"uss-type, Hardy-type,
Gagliardo-Nirenberg-type, reverse Minkowski and reverse H\"older
inequalities (see, e.g., \cite{Dahmani,Han,Mubeen,Nisar,Rahman,Rahman2,Rashid,Sawano,Set}).

In this work we obtain new Jensen-type inequalities for convex and $m$-convex functions, and we apply them to the generalized Riemann-Liouville-type
integral operators defined in \cite{BCRS}, which include most of known Riemann-Liouville-type integral operators.

\section{Preliminaries}

One of the first operators that can be called fractional is the Riemann-Liouville fractional derivative of order $\alpha \in \mathbb{C}$, with $Re(\alpha)> 0$, defined as follows (see \cite{GM}).

\begin{definition} \label{d:RL}
	Let $a < b$ and $f \in L^{1}((a,b);\mathbb{R})$.
	The \emph{right and left side Riemann-Liouville fractional integrals of order} $\alpha$, with $Re(\alpha)> 0$, are defined, respectively, by
	\begin{equation}\label{e:RL+}
		_{  }^{ RL }\!{ { J }_{ { a }^{ + } }^{ \alpha  } }f(t)=\frac { 1 }{ \Gamma (\alpha ) } \int _{ a }^{ t }{ { (t-s) }^{ \alpha -1 }f(s)\,ds },
	\end{equation}
	and
	\begin{equation}\label{e:RL-}
		_{  }^{ RL }\!{ { J }_{ { b }^{ - } }^{ \alpha  } }f(t)=\frac { 1 }{ \Gamma (\alpha ) } \int _{ t }^{ b }{ { (s-t) }^{ \alpha -1 }f(s)\,ds },
	\end{equation}
	with $t \in (a,b)$.
\end{definition}

When $\a \in (0,1)$, their corresponding \emph{Riemann-Liouville fractional derivatives} are given by
$$
\begin{aligned}
	\big({ _{  }^{ RL }\!D_{ a^{ + } }^{ \alpha  } }f\big)(t)
	& =\frac { d }{ dt } \left( { _{  }^{ RL }\!{ { J }_{ { a }^{ + } }^{ 1-\alpha  } }f(t) } \right)
	=\frac { 1 }{ \Gamma (1-\alpha ) } \, \frac { d }{ dt } \int _{ a }^{ t }{ { \frac { f(s) }{ (t-s)^{ \alpha  } }  }\,ds },
	\\
	\big({ _{  }^{ RL }\!D_{ b^{ - } }^{ \alpha  } }f\big)(t)
	& =-\frac { d }{ dt } \left( { _{  }^{ RL }\!{ { J }_{ { b }^{ - } }^{ 1-\alpha  } }f(t) } \right)
	=-\frac { 1 }{ \Gamma (1-\alpha ) } \,\frac { d }{ dt } \int _{ t }^{ b }{ { \frac { f(s) }{ (s-t)^{ \alpha  } }  }\,ds }.
\end{aligned}
$$

Other definitions of fractional operators are the following ones.

\begin{definition} \label{d:H}
	Let $a < b$ and $f \in L^{1}((a,b);\mathbb{R})$.
	The \emph{right and left side Hadamard fractional integrals of order} $\alpha$, with Re($\alpha)>0$, are defined, respectively, by
	\begin{equation}\label{e:H+}
		{ { H }_{ { a }^{ + } }^{ \alpha  } }f(t)=\frac { 1 }{ \Gamma (\alpha ) } \int _{ a }^{ t }{ { \Big(\log\frac { t }{ s } \,\Big) }^{ \alpha -1 }\frac { f(s) }{s } \, ds } ,
	\end{equation}
	and
	\begin{equation}\label{e:H-}
		{ { H }_{ { b }^{ - } }^{ \alpha  } }f(t)=\frac { 1 }{ \Gamma (\alpha ) } \int _{ t }^{ b }{ { \Big(\log\frac { s }{t } \,\Big) }^{ \alpha -1 }\frac { f(s) }{ s } \, ds } ,
	\end{equation}
	with $t \in (a,b)$.
\end{definition}

%Note that
%$$
%\frac { -\Gamma (\alpha +1) }{ B(\alpha ,1-\alpha ) }
%= \frac { -\Gamma (\alpha +1) }{ \G(\alpha )\G(1-\alpha ) }
%= \frac { -\a \, \Gamma (\alpha) }{ \G(\alpha )(-\a)\G(-\alpha ) }
%= \frac { 1 }{ \G(-\alpha ) } \,.
%$$

When $\a \in (0,1)$, \emph{Hadamard fractional derivatives} are given by the following expressions:
$$
\begin{aligned}
	{ \big( {  }^{ H }\!{ D_{ a^{ + } }^{ \alpha  } } }f\big)(t)
	& =t\,\frac { d }{ dt } \big( { { H }_{ { a }^{ + } }^{ 1-\alpha  } }f(t) \big)
	=\frac { 1 }{ \Gamma(1-\alpha ) } \, t \, \frac { d }{ dt }\int _{ a }^{ t }{ { \Big(\log \frac { t }{ s } \Big) }^{ -\alpha }\frac { f(s) }{ s } \,ds } ,
	\\
	{ \big( {  }^{ H }\!{ D_{ b^{ - } }^{ \alpha  } } }f\big)(t)
	& =-t\,\frac { d }{ dt } \big( { { { H }_{ { b }^{ - } }^{ 1-\alpha  } }f(t) } \big)
	= \frac {-1}{ \Gamma (1-\alpha ) } \, t \, \frac { d }{ dt } \int _{ t }^{ b }{ { \Big(\log\frac { s }{ t } \Big) }^{ -\alpha }\frac { f(s) }{ s } \,ds } ,
\end{aligned}
$$
with $t \in (a,b)$.

\medskip

\begin{definition} \label{d:20}
	Let $0 < a < b$,
	$g:[a, b] \rightarrow \mathbb{R}$ an increasing positive function on $(a,b]$ with continuous derivative on $(a,b)$,
	$f:[a, b]\rightarrow \mathbb{R}$ an integrable function, and $\alpha \in (0,1)$ a fixed real number.
	The right and left side fractional integrals in \cite{KMS} of order $\alpha$ of $f$ with respect to $g$ are defined,
	respectively, by
	\begin{equation}\label {e:fg+}
		I_{ g,a^+ }^{ \alpha  }f(t)=\frac { 1 }{ \Gamma (\alpha ) } \int _{ a }^{ t } \frac { g'(s)f(s) }{ { \big( g(t)-g(s) \big)  }^{ 1-\alpha  } } \,ds,
	\end{equation}
	and
	\begin{equation}\label {e:fg-}
		I_{g,b^- }^{ \alpha  }f(t)=\frac { 1 }{ \Gamma (\alpha ) } \int _{ t }^{ b } \frac { g'(s)f(s) }{ { \big( g(s)-g(t) \big)  }^{ 1-\alpha  } } \,ds,
	\end{equation}
	with $t \in (a,b)$.
	%It will be very easy for the reader to build the kernel $T$ in this case.?????
\end{definition}

There are other definitions of integral operators in the global case, but they are slight modifications of the previous ones.

%%%%%%%%%%%%%%%%%%%%%%%%%%%%%%%%%%%%%%%%%%
\section{General fractional integral of Riemann-Liouville type}

Now, we give the definition of a general fractional integral in \cite{BCRS}.

\begin{definition} \label{d:01}
	Let $a<b$ and $\a \in \RR^+$.
	Let $g:[a, b] \rightarrow \mathbb{R}$ be a positive function on $(a,b]$ with continuous positive derivative on $(a,b)$,
	and $G:[0, g(b)-g(a)]\times (0,\infty)\rightarrow \mathbb{R}$ a continuous function which is positive on $(0, g(b)-g(a)]\times (0,\infty)$.
	%be a positive absolutely continuous function on $I$ for each value of $\a$.
	Let us define the function $T:[a, b] \times [a, b] \times (0,\infty)\rightarrow \mathbb{R}$ by
	$$
	T(t,s,\alpha)
	= \frac{ G\big( |g(t)-g(s)| ,\alpha \big) }{ g'(s) } \,.
	$$
	The \emph{right and left integral operators}, denoted respectively by $J_{T,a^+}^\a$ and $J_{T,b^-}^\a$,
	are defined for each measurable function $f$ on $[a,b]$ as
	\begin{equation}\label {e:oig+}
		J_{T,a^+}^\a f(t)=\int _{ a }^{ t } \frac { f(s) }{ T(t,s,\alpha ) } \, ds,
	\end{equation}
	\begin{equation}\label {e:oig-}
		J_{T,b^-}^\a f(t)=\int _{ t }^{ b } \frac { f(s) }{ T(t,s,\alpha ) } \, ds,
	\end{equation}
	with $t \in [a,b]$.
	
	We say that $f \in L_T^1[a,b]$ if $J_{T,a^+}^\a |f|(t), J_{T,b^-}^\a |f|(t) < \infty$ for every $t \in [a,b]$.
\end{definition}

\

Note that these operators generalize the integral operators in Definitions \ref{d:RL}, \ref{d:H} %, \ref{d:K}
and \ref{d:20}:

\medskip

$(A)$ If we choose
$$
g(t)=t,
\quad G( x ,\alpha ) = \Gamma(\a)\, x^{ 1-\alpha},
\quad T(t,s,\alpha )=\Gamma (\alpha )\, |t-s|^{ 1-\alpha},
$$
then $J_{T,a^+}^\a$ and $J_{T,b^-}^\a$
are the right and left Riemann-Liouville fractional integrals $_{  }^{ RL }\!{ { J }_{ { a }^{ + } }^{ \alpha  } }$ and $_{  }^{ RL }\!{ { J }_{ { b }^{- } }^{ \alpha  } }$
in \eqref{e:RL+} and \eqref{e:RL-}, respectively.
Its corresponding right and left Riemann-Liouville fractional derivatives are
$$
\begin{aligned}
	\big({ _{  }^{ RL }\!D_{ a^{ + } }^{ \alpha  } }f\big)(t)
	=\frac { d }{ dt } \left( { _{  }^{ RL }\!{ { J }_{ { a }^{ + } }^{ 1-\alpha  } }f(t) } \right),
	\quad
	\big({ _{  }^{ RL }\!D_{ b^{ - } }^{ \alpha  } }f\big)(t)
	=-\frac { d }{ dt } \left( { _{  }^{ RL }\!{ { J }_{ { b }^{ - } }^{ 1-\alpha  } }f(t) } \right).
\end{aligned}
$$

\medskip

$(B)$ If we choose
$$
g(t)= \log t,
\quad G( x ,\alpha ) = \Gamma(\a)\, x^{ 1-\alpha},
\quad T(t,s,\alpha )=\Gamma (\alpha )\, t \, \Big|\log \frac{t}{s}\Big|^{ 1-\alpha},
$$
then $J_{T,a^+}^\a$ and $J_{T,b^-}^\a$
are the right and left Hadamard fractional integrals ${ H }_{ { a }^{ + } }^{\alpha}$ and ${ H }_{ { b }^{ - } }^{\alpha}$
in \eqref{e:H+} and \eqref{e:H-}, respectively.
Its corresponding right and left Hadamard fractional derivatives are
$$
\begin{aligned}
	{ \big( {  }^{ H }\!{ D_{ a^{ + } }^{ \alpha  } } }f\big)(t)
	=t\,\frac { d }{ dt } \big( { { H }_{ { a }^{ + } }^{ 1-\alpha} }f(t) \big) ,
	\quad
	{ \big( {  }^{ H }\!{ D_{ b^{ - } }^{ \alpha  } } }f\big)(t)
	=-t\,\frac { d }{ dt } \big( { { { H }_{ { b }^{ - } }^{ 1-\alpha} }f(t) } \big).
\end{aligned}
$$

\medskip

$(C)$ If we choose a function $g$ with the properties in Definition \ref{d:01} and
$$
G( x ,\alpha ) = \Gamma(\a)\, x^{ 1-\alpha},
\quad T(t,s,\alpha ) = \Gamma(\a)\, \frac{\, |\, g(t)-g(s) |^{ 1-\alpha}}{g'(s)} \,,
$$
then $J_{T,a^+}^\a$ and $J_{T,b^-}^\a$
are the right and left %Kilbas-Marichev-Samko
fractional integrals $I_{ g,a^+ }^{ \alpha  }$ and $I_{ g,b^- }^{ \alpha  }$
in \eqref{e:fg+} and \eqref{e:fg-}, respectively.

\medskip

%Obviously, we can define the lateral derivative operators (right and left) in the case of our generalized derivative, for this it is sufficient to consider them from the corresponding integral operator. To do this, just make use of the fact that if $f$ is differentiable, then $N_{F}^{\alpha}f(t)=F(t,\alpha)f'(t)$ where $f'(t)$ is the ordinary derivative. For the right derivative we have $\left( { N }_{ F,a+ }^{ \alpha  }f \right) (t)={ N }_{ F }^{ \alpha  }\left[ { J }_{ T,a+ }^{ \alpha  }(f)(t) \right] =\frac { d }{ dx } \left[ { J }_{ T,a+ }^{ \alpha  }(f)(t) \right] F(x,\alpha )$, similarly to the left. \newline
%
%7) A k-analogue of above definition is defined in \cite{KFNUK} (also see \cite{F}), under the same assumptions on function $g$
%
%\begin{equation}\label {e:kfg+}
%I_{ g,a+ }^{ \alpha ,k }(f)(t)=\frac { 1 }{ \Gamma (\alpha ) } \int _{ a }^{ t } \frac { g'(s)f(s) }{ { \left[ g(t)-g(s) \right]  }^{ 1-\frac { \alpha  }{ k }  } } ds,\quad t>a,
%\end{equation}
%
%similarly the right lateral derivative is defined as well
%
%\begin{equation}\label {e:kfg-}
%I_{ g,b- }^{ \alpha ,k }(f)(t)=\frac { 1 }{ \Gamma (\alpha ) } \int _{ t }^{ b } \frac { g'(s)f(s) }{ { \left[ g(s)-g(t) \right]  }^{ 1-\frac { \alpha  }{ k }  } } ds,\quad t<b.
%\end{equation}
%
%The corresponding differential operator is also very easy to obtain.
%
%8) We can define the function space $L_{\alpha}^{ p}[a,b]$ as the set of functions over $[a,b]$ such that $({J}_{T,a+}^{\alpha}{[f(t)]^p}(b))<+\infty$.
%\end{remark}

\begin{definition} \label{d:0}
	Let $a<b$ and $\a \in \RR^+$.
	Let $g:[a, b] \rightarrow \mathbb{R}$ be a positive function on $(a,b]$ with continuous positive derivative on $(a,b)$,
	and $G:[0, g(b)-g(a)]\times (0,\infty)\rightarrow \mathbb{R}$ a continuous function which is positive on $(0, g(b)-g(a)]\times (0,\infty)$.
	For each function $f \in L_T^1[a,b]$, its \emph{right and left generalized derivative of order} $\alpha $ are defined, respectively,  by
	\begin{equation} \label{e:d0}
		\begin{aligned}
			D_{T,a^+ }^{\alpha} f(t)
			& = \frac{1}{g'(t)} \, \frac { d }{ dt } \left( { J }_{ T,a^+ }^{ 1-\alpha  }f (t) \right),
			\\
			D_{T,b^- }^{\alpha} f(t)
			& = \frac{-1}{g'(t)} \, \frac { d }{ dt } \left( { J }_{ T,b^- }^{ 1-\alpha  }f (t) \right).
		\end{aligned}
	\end{equation}
	for each $t \in (a,b)$.
\end{definition}

Note that if we choose
$$
g(t)=t,
\quad G( x ,\alpha ) = \Gamma(\a)\, x^{ 1-\alpha},
\quad T(t,s,\alpha )=\Gamma (\alpha )\, |t-s|^{ 1-\alpha},
$$
then $D_{T,a^+ }^{\alpha}f(t)= \, _{ }^{ RL }\!{ { D}_{a^+ }^{ \alpha  } }f(t)$
and $D_{T,b^- }^{\alpha}f(t)= \, _{ }^{ RL }\!{ { D}_{b^-}^{ \alpha  } }f(t)$.
Also, we can obtain Hadamard and others fractional derivatives
as particular cases of this generalized derivative.

\bigskip

\section{Jensen-type inequalities for $m$-convex functions}

The property of $m$-convexity for functions on $[0,b]$, $b>0$ was introduced in \cite{Toader} as an intermediate property between the usual convexity and starshaped property.  Since then
many properties, especially inequalities, have been obtained for them (cf. \cite{Dragomir,KPR,Lara,PAA}). One of the classical integral inequalities frequently studied in this setting is Jensen's inequality, which relates the value of a convex function of an integral to the integral of the convex function. It was proved in 1906 \cite{Jensen}, and it can be stated as follows:

\medskip

Let $\mu$ be a probability measure on the space $X$. If $f: X \rightarrow (a,b)$ is $\mu$-integrable and $\varphi$ is a convex function on $(a,b)$, then
$$
\varphi \left(\int _{X}f\,d\mu \right)\leq \int _{X}\varphi \circ f\,d\mu \,.
$$

\begin{definition}
 \label{m:T}
Let $I\subseteq \mathbb{R}$ be an interval containing the zero, and let $m\in (0,1]$.
A function $\varphi:I \rightarrow \mathbb{R}$  is said to be $m$-convex if the inequality
\begin{equation}
\label{eq:5}
\varphi(t x+m(1-t) y) \leq t \varphi(x)+m(1-t) \varphi(y),
\end{equation}
holds for every pair of points $x,y\in I$ and every coefficient $t\in [0,1]$.
\end{definition}

If $m \in (0,1)$, then the hypothesis $0 \in I$ guarantees that $t x+m(1-t) y \in I$.

\medskip

It is clear that taking $m = 1$ in Definition \ref{m:T} we recover the concept of classical convex functions on $I$.
Note that in this case it is not necessary the hypothesis $0 \in I$, since $t x+(1-t) y \in I$ for every $x, y \in I$.

\medskip

Note that if we choose the coefficient $t=0$ in \eqref{eq:5}, we get the inequality $\varphi(m y) \leq m \varphi( y)$.
%Thus, as pointed out in \cite{PAA}, the case $x=my$ can be omitted in Definition \ref{m:T}.

\medskip

Also, Definition \ref{m:T} is equivalent to
\begin{equation}
\label{eq:2}
\varphi(mt x+(1-t) y) \leq mt \varphi(x)+(1-t) \varphi(y),
\end{equation}
for all $x,y\in I$ and $t\in [0,1]$.

\medskip

The following  discrete Jensen-type inequality for $m$-convex functions was established in  \cite[Theorem 3.2]{PAA}:

\begin{theorem}
\label{prop:2}
Let $I\subseteq \mathbb{R}$ be an interval containing the zero, and let $\sum_{k=1}^{n} w_{k} x_{k}$ be a convex combination of points $x_{k}\in I$ with coefficients $w_{k} \in [0,1]$. If $\varphi$ is an $m$-convex function on $I$, with $m\in (0,1]$, then
\begin{equation}
\label{eq:3}
\varphi\Big(m\sum_{k=1}^{n}w_{k} x_{k}\Big)\leq m \sum_{k=1}^{n}w_{k}\varphi(x_{k}).
\end{equation}
\end{theorem}

This inequality is a discrete version of the following one for continuous $m$-convex functions \cite[Corollary 4.2]{PAA}:

\begin{theorem}
\label{prop:2general}
Let $\mu$ be a probability measure on the space $X$.
If $I\subseteq \mathbb{R}$ is an interval containing the zero, $f: X \rightarrow I$ is $\mu$-integrable and $\varphi$ is a continuous $m$-convex function on $I$,
with $m\in (0,1]$, then
\begin{equation}
\varphi \left(m \int _{X}f\,d\mu \right)\leq m \int _{X}\varphi \circ f\,d\mu \,.
\end{equation}
\end{theorem}

The following  discrete Jensen-type inequality for convex functions appears in  \cite[Theorem 1.2]{Mercer}:

\begin{theorem}
\label{t:1}
Let $x_1 \le x_2 \le \dots \le x_n$ and let $\{w_k\}_{k=1}^{n}\!$ be positive weights whose sum is $1$.
If $\varphi$ is a convex function on $[x_1,x_n]$, then
$$
\varphi \Big( x_1+x_n - \sum_{k=1}^{n} w_k x_k \Big)
\le \varphi(x_1) + \varphi(x_n) - \sum_{k=1}^{n} w_k \varphi (x_k) .
$$
\end{theorem}

Our purpose is to prove continuous versions of the above discrete inequality in the setting of $m$-convexity
(see Theorems \ref{t:Jensen2m} and \ref{t:Jensen3}).
Before stating  such a result, we require some properties of the $m$-convex functions.

\begin{lemma}
\label{le:1}
Let $I\subseteq \mathbb{R}$ be an interval containing the zero, and let $\varphi$ be an $m$-convex function on $I$ with $m \in (0,1]$.
For $\{x_k\}_{k=1}^{n} \subset I\!$ such that $x_{1}\leq x_{2}\leq \cdots\leq x_{n}$, the following inequalities hold:
\begin{equation}
\label{eq:1}
\varphi\left(x_{1}+mx_{n}-mx_{k}\right) \leq \varphi\left(x_{1}\right)+m\varphi\left(x_{n}\right)-\varphi\left(mx_{k}\right), \quad 1 \leq k \leq n.
\end{equation}
%If $m=1$ these inequalities also hold if we remove the hypothesis $0 \in I$.
\end{lemma}

\begin{proof}
Let us consider $y_{k}=x_{1}+mx_{n}-mx_{k}$. Then  $x_{1}+mx_{n}=y_{k}+mx_{k}$ and so,
the pairs $x_{1},mx_{n}$  and  $y_{k},mx_{k}$ have the same mid-point.
Since $x_1 \le y_k$ and $mx_k \le mx_n$,
we have $x_1 \le y_k,mx_k \le mx_n$ and
there exists $\lambda\in [0,1]$ such that
$$
\begin{aligned}
&m x_{k}=\lambda x_{1}+m(1-\lambda) x_{n}, \\
&y_{k}=(1-\lambda) x_{1}+m\lambda x_{n},
\end{aligned}
$$
for $ 1 \leq k \leq n$. From Definition \ref{m:T} and its equivalent form  \eqref{eq:2} we get
$$
\begin{aligned}
\varphi\left(y_{k}\right) & \leq m\lambda \varphi\left(x_{n}\right) +(1-\lambda) \varphi\left(x_{1}\right)\\
&=\varphi\left(x_{1}\right)+m\varphi\left(x_{n}\right)-\big[\lambda \varphi\left(x_{1}\right)+m(1-\lambda) \varphi\left(x_{n}\right)\big] \\
& \leq \varphi\left(x_{1}\right)+m\varphi\left(x_{n}\right)-\varphi\left(\lambda x_{1}+m(1-\lambda) x_{n}\right) \\
&=\varphi\left(x_{1}\right)+m\varphi\left(x_{n}\right)-\varphi\left(mx_{k}\right),
\end{aligned}
$$
and \eqref{eq:1} follows.

Note that since $m \in (0,1]$, the hypothesis $0 \in I$ guarantees that $mx_{k}\in I$.
%If $m \in (0,1)$, then the hypothesis $0 \in I$ guarantees that $mx_{k}\in I$.
%
%If $m=1$, then $mx_{k}=x_k \in I$ directly, and so, we can remove the hypothesis $0 \in I$.
\end{proof}

\medskip

The following two results generalize Theorem \ref{t:1} in the setting of $m$-convexity.

\begin{theorem} \label{t:3}
Let $I\subseteq \mathbb{R}$ be an interval containing the zero, let $\{x_{k}\}_{k=1}^{n}\subset I$ with $x_1 \le x_2 \le \dots \le x_n$
 and let $\{w_k\}_{k=1}^{n}\!$ be positive weights whose sum is $1$.
If $\varphi$ is an $m$-convex function on $I$, with $m\in (0,1]$, then
\begin{equation}
\label{eq:4}
\varphi\Big(mx_{1}+m^{2}x_{n}-m^{2}\sum_{k=1}^{n}w_{k} x_{k}\Big)\leq m \varphi(x_{1})+m^{2}\varphi(x_{n})-m\sum_{k=1}^{n}w_{k}\varphi(mx_{k}).
\end{equation}
\end{theorem}

\begin{remark} \label{r:3}
Theorem \ref{t:1} gives that if $m=1$, the inequality in Theorem \ref{t:3}
also holds if we remove the hypothesis $0 \in I$.
\end{remark}

\begin{proof} First, note that
$$x_{1}+mx_{n}-m\sum_{k=1}^{n}w_{k} x_{k}=\sum_{k=1}^{n}w_{k}(x_{1}+mx_{n}-mx_{k}),$$
and thus
$$mx_{1}+m^{2}x_{n}-m^{2}\sum_{k=1}^{n}w_{k} x_{k}=m\sum_{k=1}^{n}w_{k}(x_{1}+mx_{n}-mx_{k}).$$
Then it follows from \eqref{eq:3} and \eqref{eq:1} that
$$
\begin{aligned}
  \varphi\Big(mx_{1}+m^{2}x_{n}-m^{2}\sum_{k=1}^{n}w_{k} x_{k}\Big) &= \varphi\Big(m\sum_{k=1}^{n}w_{k}(x_{1}+mx_{n}-mx_{k})\Big) \\
   &\leq m\sum_{k=1}^{n}w_{k}\varphi\left(x_{1}+mx_{n}-mx_{k}\right) \\
 &\leq m\sum_{k=1}^{n}w_{k}\left(\varphi(x_{1})+m\varphi(x_{n})-\varphi(mx_{k})\right) \\
 &=  m\varphi(x_{1}) +m^{2}\varphi(x_{n})-m\sum_{k=1}^{n}w_{k}\varphi(mx_{k}),
\end{aligned}
$$
and this concludes the proof of the inequality.

\smallskip

Let us check that the hypothesis $0 \in I$ guarantees that $mx_{1}+m^{2}x_{n}-m^{2}\sum_{k=1}^{n}w_{k} x_{k} \in I$:

\smallskip

Assume that $I=[a,b]$. Then
$$
\begin{aligned}
mx_{1}+m^{2}x_{n}-m^{2}\sum_{k=1}^{n}w_{k} x_{k}
& \ge mx_{1}+m^{2}x_{n}-m^{2}\sum_{k=1}^{n}w_{k} x_{n}
%\\
%& = mx_{1}+m^{2}x_{n}-m^{2} x_{n}
\\
& = mx_{1}
\ge \min\{0,x_1\} \ge a.
\end{aligned}
$$
Also,
$$
\begin{aligned}
mx_{1}+m^{2}x_{n}-m^{2}\sum_{k=1}^{n}w_{k} x_{k}
& \le mx_{1}+m^{2}x_{n}-m^{2}\sum_{k=1}^{n}w_{k} x_{1}
\\
& = mx_{1}-m^{2} x_{1}+m^{2}x_{n}.
\end{aligned}
$$
If $x_1 \le 0$, then $mx_{1}-m^{2} x_{1}\le 0$ and so,
$$
\begin{aligned}
mx_{1}+m^{2}x_{n}-m^{2}\sum_{k=1}^{n}w_{k} x_{k}
& \le mx_{1}-m^{2} x_{1}+m^{2}x_{n}
\le m^{2}x_{n}
\le \max\{0,x_n\}
\le b.
\end{aligned}
$$
Assume now that $x_1 > 0$.
Let us consider the function $v(t)=tx_{1}-t^{2} x_{1}+t^{2}x_{n}$.
Since $v'(t)= x_{1} +2t(x_{n}-x_1) > 0$ and $m \in (0,1]$, we have
$$
\begin{aligned}
v(m) \le v(1)
& = x_n \le b,
\\
mx_{1}+m^{2}x_{n}-m^{2}\sum_{k=1}^{n}w_{k} x_{k}
& \le mx_{1}-m^{2} x_{1}+m^{2}x_{n}
= v(m)
\le b.
\end{aligned}
$$
%
%If $m=1$, then
%$$
%\begin{aligned}
%x_{1}+x_{n}-\sum_{k=1}^{n}w_{k} x_{k}
%& \ge x_{1}+ x_{n}- \sum_{k=1}^{n}w_{k} x_{n}
%= x_{1} \ge a,
%\\
%x_{1}+x_{n}-\sum_{k=1}^{n}w_{k} x_{k}
%& \le x_{1}+ x_{n}- \sum_{k=1}^{n}w_{k} x_{1}
%= x_{n} \le b,
%\end{aligned}
%$$
%and so, we can remove the hypothesis $0 \in I$.

If $I$ is not a closed interval $[a,b]$, a similar argument gives the result.
\end{proof}

%\begin{remark}
%Note that $mx_{1}+m^{2}x_{n}-m^{2}\sum_{k=1}^{n}w_{k} x_{k} \in I$:
%$$
%a
%$$
%\end{remark}

\medskip

If $\varphi$ is a continuous $m$-convex function, we can obtain the following improvement of Theorem \ref{t:3}.

\begin{theorem} \label{t:3m}
Let $a \le 0 \le b$, let $\{y_{k}\}_{k=1}^{n}\subset [a,b]$ and let $\{w_k\}_{k=1}^{n}\!$ be positive weights whose sum is $1$.
If $\varphi$ is a continuous $m$-convex function on $[a,b]$, with $m\in (0,1]$, then
\begin{equation}
\label{eq:7}
\varphi\Big(ma+m^{2}b-m^{2}\sum_{k=1}^{n}w_{k} y_{k}\Big)\leq m \varphi(a)+m^{2}\varphi(b)-m\sum_{k=1}^{n}w_{k}\varphi(my_{k}).
\end{equation}
\end{theorem}

\begin{proof}
If we consider $0<\e <1$, $y_0=a$, $y_{n+1}=b$, $w_k'=(1-\e)w_k$ $(1 \le k \le n)$, $w_0'=\e/2$ and $w_{n+1}'=\e/2$, then $\sum_{k=0}^{n+1} w_k'=1$
and Theorem \ref{t:3} gives
$$
\begin{aligned}
\varphi \Big( ma & +m^{2}b - \frac{m^{2}\e}2\,a - \frac{m^{2}\e}2\,b - m^{2}\sum_{k=1}^{n} (1-\e)w_k y_k \Big)
\\
& \le m\varphi(a) + m^{2}\varphi(b) - \frac{m\e}2\,\varphi(ma) - \frac{m\e}2\,\varphi(mb)
 - m \sum_{k=1}^{n} (1-\e)w_k \varphi (my_k) .
\end{aligned}
$$
Since $\varphi$ is a continuous function on $[a,b]$, if we take $\e \to 0^+$, we obtain
\eqref{eq:7}.
\end{proof}

\medskip

Next, we present a continuous version of the above discrete inequality.

\begin{theorem} \label{t:Jensen2m}
Let $\mu$ be a probability measure on the space $X$ and $a \le 0\le b$ real constants. If $f: X \rightarrow [a,b]$ is a measurable function and $\varphi$ is a continuous $m$-convex function on $[a,b]$, with $m\in (0,1]$, then
$f$ and $\varphi (mf)$ are $\mu$-integrable functions and
\begin{equation}
\label{eq:6}
\varphi \Big( ma+m^{2}b -m^{2} \int _{X} f\,d\mu \Big)
\le m\varphi(a) + m^{2}\varphi(b) -m \int _{X}\varphi (mf)\,d\mu .
\end{equation}
If $m=1$, this inequality also holds if we remove the hypothesis $0 \in [a,b]$.
\end{theorem}

\begin{proof}
Since $a \leq f \leq b$ and $\varphi$ is a continuous function on $[a,b]$, we have that $f$ and $\varphi (mf)$ are bounded measurable functions on $X$.
And using that $\mu$ is a probability measure on $X$, we conclude that $f$ and $\varphi (mf)$ are $\mu$-integrable functions.

\medskip

For each $n \ge 1$ and $0 \le k \le 2^n$, let us consider the sets
$$
E_{n,k}
= \big\{ x \in X: \, a+k2^{-n}(b-a) \le f(x) < a+(k+1)2^{-n}(b-a) \big\}.
$$
Since $f$ is a measurable function satisfying $a \le f \le b$, we have that
$\{E_{n,k}\}_{k=0}^{2^n}$ are pairwise disjoint measurable sets and $X = \cup_{k=0}^{2^n} E_{n,k}$ for each $n$.
Thus,
$$
\sum_{k=0}^{2^n} \mu(E_{n,k})=1
$$
for each $n$.

\medskip

Since $f$ is a measurable function satisfying $a \le f \le b$ and $\{E_{n,k}\}_{k=0}^{2^n}$ is a partition of $X$, the sequence of simple functions
$$
f_n
= \sum_{k=0}^{2^n} \big( a+k2^{-n}(b-a) \big) \chi_{E_{n,k}}
$$
satisfies $a \le f_n \le b$ and $f-2^{-n}(b-a) < f_n \le f$ for every $n$ and so,
$$
\lim_{n \to \infty} f_n = f.
$$
Note that
$$
\int _{X} f_n\,d\mu
= \sum_{k=0}^{2^n} \big( a+k2^{-n}(b-a) \big) \mu(E_{n,k}) .
$$

\medskip

Since $\{E_{n,k}\}_{k=0}^{2^n}$ is a partition of $X$, we have
$$
\begin{aligned}
\varphi (mf)
& = \sum_{k=0}^{2^n} \varphi \big(ma+mk2^{-n}(b-a) \big) \chi_{E_{n,k}},
\\
\int_{X} \varphi (mf)\,d\mu
& = \sum_{k=0}^{2^n} \varphi \big( ma+mk2^{-n}(b-a) \big) \mu(E_{n,k}) .
\end{aligned}
$$
Hence, Theorem \ref{t:3m} gives
\begin{equation}
\label{eq:8}
\varphi \Big( ma+m^{2}b -m^{2} \int _{X} f_n\,d\mu \Big)
\le m\varphi(a) + m^{2}\varphi(b) - m\int_{X} \varphi (mf)\,d\mu.
\end{equation}

If $m=1$, Theorem \ref{t:1} gives the above inequality without the hypothesis $0 \in [a,b]$.

\smallskip

Since $a \le f_n \le b$ for every $n$, $\mu$ is a finite measure and $\displaystyle{\lim_{n \to \infty} f_n = f}$,
dominated convergence theorem gives
$$
\lim_{n \to \infty} \int _{X} f_n\,d\mu
= \int _{X} f \,d\mu.
$$

If $m \in (0,1)$, then the hypothesis $0 \in [a,b]$ guarantees that $ma+m^{2}b -m^{2} \int _{X} f_n\,d\mu \in [a,b]$:

\medskip

Since $a \le 0 \le b$, we have
$$
\begin{aligned}
ma+m^{2}b -m^{2} \int _{X} f_n\,d\mu
& \le ma+m^{2}b -m^{2} a
= m(1-m)a+m^2b
\le m^2b \le b,
\\
ma+m^{2}b -m^{2} \int _{X} f_n\,d\mu
& \ge ma+m^{2}b -m^{2} b
= ma \ge a.
\end{aligned}
$$

If $m=1$, then
$$
\begin{aligned}
a+b - \int _{X} f_n\,d\mu
& \le a+b - a
= b,
\\
a+b - \int _{X} f_n\,d\mu
& \ge a+b - b
= a,
\end{aligned}
$$
and so, we do not need the hypothesis $0 \in [a,b]$.

\medskip

Since
$$
a \le ma+m^{2}b -m^{2} \int _{X} f_n\,d\mu \le b
$$
for every $n$ and $m \in (0,1]$,
and $\varphi$ is a continuous function on $[a,b]$,
$$
\lim_{n \to \infty} \varphi \Big( ma+m^{2}b -m^{2} \int _{X} f_n\,d\mu \Big)
= \varphi \Big( ma+m^{2}b -m^{2} \int _{X} f\,d\mu \Big).
$$

Since $a \le mf_n \le b$ for every $n$, $\displaystyle{\lim_{n \to \infty} f_n = f}$
and $\varphi$ is a continuous function on $[a,b]$,
$\displaystyle{\lim_{n \to \infty} \varphi (mf_{n}) = \varphi (mf)}$.

\medskip

Again, from the continuity of  $\varphi$  on $[a,b]$, there exists a positive constant $K$ with $|\varphi| \le K$ on $[a,b]$ and so, $|\varphi (mf_n)| \le K$ for every $n$.

\medskip

In view of the finiteness of  $\mu$, dominated convergence theorem guarantees that
$$
\lim_{n \to \infty} \int _{X} \varphi (mf)\,d\mu
= \int _{X} \varphi (mf)\,d\mu \,.
$$

Combining the foregoing facts with \eqref{eq:8}, we obtain \eqref{eq:6}.
\end{proof}

\medskip

If $m=1$, it is possible to improve Theorem \ref{t:Jensen2m},
by removing the hypothesis of continuity.

\begin{theorem}
\label{t:Jensen3}
Let $\mu$ be a probability measure on the space $X$ and $a \le b$ real constants.
If $f: X \rightarrow [a,b]$ is a measurable function and $\varphi$ is a convex function on $[a,b]$, then
$f$ and $\varphi \circ f$ are $\mu$-integrable functions and
$$
\varphi \Big( a+b - \int _{X} f\,d\mu \Big)
\le \varphi(a) + \varphi(b) - \int _{X}\varphi \circ f\,d\mu .
$$
\end{theorem}

\begin{proof}
Since $\varphi$ is a convex function on $[a,b]$, $\varphi$ is continuous on $(a,b)$ and there exist the limits
$$
\lim_{s \to a^+} \varphi(s) ,
\qquad
\lim_{s \to b^-} \varphi(s) .
$$
Define $\varphi^*$ as follows
$$
\varphi^*(t) =
\begin{cases}
\varphi(t) \quad & \text{if }\, t \in (a,b),
\\
\lim_{s \to a^+} \varphi(s) \quad & \text{if }\, t = a,
\\
\lim_{s \to b^-} \varphi(s) \quad & \text{if }\, t = b.
\end{cases}
$$
Hence, $\varphi^*$ is a continuous convex function on $[a,b]$, and by Theorem \ref{t:Jensen2m} with $m=1$, we have
$$
\varphi^* \Big( a+b - \int _{X} f\,d\mu \Big)
\le \varphi^*(a) + \varphi^*(b) - \int _{X}\varphi^* \circ f\,d\mu.
$$

Assume that $f=a$ $\mu$-a.e. or $f=b$ $\mu$-a.e.;
in the first case,
$$
\begin{aligned}
\varphi \Big( a+b - \int _{X} f\,d\mu \Big)
& = \varphi ( a+b - a)
= \varphi (b)
\\
& = \varphi(a) + \varphi(b) - \varphi(a)
= \varphi(a) + \varphi(b) - \int _{X}\varphi \circ f\,d\mu \,;
\end{aligned}
$$
in the second case,
$$
\begin{aligned}
\varphi \Big( a+b - \int _{X} f\,d\mu \Big)
& = \varphi ( a+b - b )
= \varphi (a)
\\
& = \varphi(a) + \varphi(b) - \varphi(b)
= \varphi(a) + \varphi(b) - \int _{X}\varphi \circ f\,d\mu \,.
\end{aligned}
$$
Otherwise, $\displaystyle{a < \int _{X} f\,d\mu < b}$ and $\displaystyle{a < a+b - \int _{X} f\,d\mu < b}$.
Consequently,
$$
\varphi \Big( a+b - \int _{X} f\,d\mu \Big)
= \varphi^* \Big( a+b - \int _{X} f\,d\mu \Big).
$$
If we define
$$
\D_a = \varphi (a) - \varphi^*  (a) \ge 0,
\qquad
\D_b = \varphi ( b ) - \varphi^*  ( b ) \ge 0,
$$
then
$$
\begin{aligned}
\varphi
& = \varphi^* + \D_a \, \chi_{\{a\}} + \D_b \, \chi_{\{b\}} ,
\\
\varphi \circ f
& = \varphi^* \circ f + \D_a \, \chi_{\{f=a\}} + \D_b \, \chi_{\{f=b\}} ,
\end{aligned}
$$
where $\chi_{A}$ is the function with value $1$ on the set $A$ and $0$ otherwise
(i.e., the characteristic function of $A$).
Hence,
$$
\int _{X}\varphi \circ f\,d\mu = \int _{X}\varphi^* \circ f\,d\mu + \D_a \, \mu\big(\{f=a\}\big) + \D_b \, \mu\big(\{f=b\}\big) ,
$$
and we have
$$
\begin{aligned}
& \varphi \Big( a+b - \int _{X} f\,d\mu \Big)
= \varphi^* \Big( a+b - \int _{X} f\,d\mu \Big)
\\
& \quad \le \varphi^*(a) + \varphi^*(b) - \int _{X}\varphi^* \circ f\,d\mu
\\
& \quad = \varphi(a) - \D_a + \varphi(b) - \D_b - \int _{X}\varphi \circ f\,d\mu  + \D_a \, \mu\big(\{f=a\}\big) + \D_b \, \mu\big(\{f=b\}\big)
\\
& \quad = \varphi(a) + \varphi(b) - \int _{X}\varphi \circ f\,d\mu  - \D_a \big[ 1 - \mu\big(\{f=a\}\big) \big] - \D_b \big[ 1 - \mu\big(\{f=b\}\big) \big]
\\
& \quad \le \varphi(a) + \varphi(b) - \int _{X}\varphi \circ f\,d\mu .
\end{aligned}
$$
\end{proof}

\medskip

Theorem \ref{t:Jensen3} has the following direct consequence.

\begin{corollary}
Let $a \le b$, let $\{y_{k}\}_{k=1}^{n}\subset [a,b]$ and let $\{w_k\}_{k=1}^{n}\!$ be positive weights whose sum is $1$.
If $\varphi$ is a convex function on $[a,b]$, then
$$
\varphi\Big(a+b-\sum_{k=1}^{n}w_{k} y_{k}\Big)\leq \varphi(a)+\varphi(b)-\sum_{k=1}^{n}w_{k}\varphi(y_{k}).
$$
\end{corollary}

\medskip

%If $\varphi$ is also a strictly increasing function,
Note that Theorem \ref{t:Jensen3} provides a kind of converse of the classical Jensen's inequality for convex functions.

\begin{proposition}
\label{p:Jensen}
Let $\mu$ be a probability measure on the space $X$ and $a \le b$ real constants.
If $f: X \rightarrow [a,b]$ is a measurable function and $\varphi$ is a convex function on $[a,b]$, then
$f$ and $\varphi \circ f$ are $\mu$-integrable functions and
$$
\varphi \Big( \int _{X} f\,d\mu \Big)
\le \int _{X}\varphi \circ f\,d\mu
\le \varphi(a) + \varphi(b) - \varphi \Big( a+b - \int _{X} f\,d\mu \Big) .
$$
\end{proposition}

%Also, Theorems \ref{prop:2general} and Theorems \ref{t:Jensen2m} provide a similar result for $m$-convex functions.
%
%\begin{proposition}
%\label{p:Jensenm}
%Let $\mu$ be a probability measure on the space $X$ and $a \le 0 \le b$ real constants.
%If $f: X \rightarrow [a,b]$ is a measurable function and $\varphi$ is a continuous $m$-convex function on $[a,b]$, with $m \in (0,1]$, then
%$f$ and $\varphi (mf)$ are $\mu$-integrable functions and
%$$
%\varphi \Big( \int _{X} f\,d\mu \Big)
%\le \int _{X}\varphi \circ f\,d\mu
%\le \varphi(a) + \varphi(b) - \varphi \Big( a+b - \int _{X} f\,d\mu \Big) .
%$$
%\end{proposition}

Theorems \ref{t:Jensen3} and \ref{t:Jensen2m}
have, respectively, the following direct consequences for general fractional integrals of Riemann-Liouville type.

\begin{proposition}
 \label{p:JensenRL}
Let $c<d$ and $a \le b$ be real constants.
If $f: [c,d] \rightarrow [a,b]$ is a measurable function, $\varphi$ is a convex function on $[a,b]$,
and
$$
	\mathbb{T}(\alpha)
	= \int _{ c }^{ d } \frac { 1 }{ T(d,s,\alpha ) } \, ds
	= \int _{ 0 }^{ g(d)-g(c) } \frac { dx}{ G(x,\alpha ) } < \infty ,
$$
then $f(s)/T\big(d,s,\alpha\big), \varphi(f(s))/T\big(d,s,\alpha\big) \in L^1[c,d]$ and
$$
\varphi \Big( a+b - \frac { 1 }{\mathbb{T}(\alpha )}\int_{c}^d \frac { f(s) }{ T\big(d,s,\alpha\big) } \, ds \Big)
\le \varphi(a) + \varphi(b) - \frac { 1 }{\mathbb{T}(\alpha )}\int_{c}^d \frac { \varphi(f(s)) }{ T\big(d,s,\alpha\big) } \, ds .
$$
\end{proposition}

\begin{proposition}
 \label{p:JensenRLm}
Let $c<d$ and $a \le 0 \le b$ be real constants.
If $f: [c,d] \rightarrow [a,b]$ is a measurable function, $\varphi$ is a continuous $m$-convex function on $[a,b]$, with $m \in (0,1]$,
and
$$
	\mathbb{T}(\alpha)
	= \int _{ c }^{ d } \frac { 1 }{ T(d,s,\alpha ) } \, ds
	= \int _{ 0 }^{ g(d)-g(c) } \frac { dx}{ G(x,\alpha ) } < \infty ,
$$
then $f(s)/T\big(d,s,\alpha\big), \varphi(mf(s))/T\big(d,s,\alpha\big) \in L^1[c,d]$ and
$$
\varphi \Big( ma+m^2b - \frac { m^2 }{\mathbb{T}(\alpha )}\!\!\int_{c}^d \!\!\frac { f(s) }{ T\big(d,s,\alpha\big) } \, ds \Big)
\le m\varphi(a) + m^2\varphi(b) - \frac { m }{\mathbb{T}(\alpha )}\!\!\int_{c}^d \!\frac { \varphi(mf(s)) }{ T\big(d,s,\alpha\big) } \, ds .
$$
\end{proposition}

\bigskip

\section*{Acknowledgments}
%We would like to thank the referees for their comments which have improved the paper.
The research of Yamilet Quintana, Jos\'e M. Rodr\'{\i}guez and Jos\'e M. Sigarreta is supported by a grant from Agencia Estatal de Investigaci\'on (PID2019-106433GB-I00 / AEI / 10.13039/501100011033), Spain.
The research of Jos\'e M. Rodr\'{\i}guez and Yamilet Quintana is supported by the Madrid Government (Comunidad de Madrid-Spain) under the Multiannual Agreement with UC3M in the line of Excellence of University Professors (EPUC3M23), and in the context of the V PRICIT (Regional Programme of Research and Technological Innovation).

\end{document}